\documentclass{article}
\usepackage[english]{babel}
\usepackage{graphicx}
\usepackage{amsmath}
\usepackage{geometry}
\usepackage{accessibility}
\usepackage{xcolor}
\usepackage{dsfont}
\usepackage{amssymb}
\usepackage{amsthm}
\usepackage{mathtools}
\usepackage{dsfont}
\usepackage[textsize=tiny]{todonotes}
\usepackage{esvect}
\usepackage{authblk}

\newtheorem{theorem}{Theorem}[section]
\newtheorem*{theorem*}{Theorem}
\newtheorem{lemma}[theorem]{Lemma}

\newtheorem{definition}[theorem]{Definition}

\theoremstyle{remark}
\newtheorem*{remark*}{Bemerkung}
\newtheorem{remark}[theorem]{Remark}
\theoremstyle{corollary}
\newtheorem{corollary}[theorem]{Corollary}
\newtheorem{example}[theorem]{Example}
\newtheorem*{example*}{Example}

\newcommand{\indic}{\mathds{1}}
\newcommand{\PP}{\mathbb{P}}

\newcommand{\hP}{\hat{P}}
\newcommand{\hY}{\hat{Y}}

\newcommand{\hpi}{\hat{\pi}}
\newcommand{\NN}{\mathbb{N}}

\newcommand{\be}{\Bar{e}}
\newcommand{\vE}{\vv{E}}
\newcommand{\bE}{E}
\newcommand{\bpi}{\Bar{\pi}}
\newcommand{\bK}{\Bar{K}}
\newcommand{\vP}{\vv{P}}
\newcommand{\vpi}{\vv{\pi}}

\newenvironment{nalign}{
    \begin{equation}
    \begin{aligned}
}{
    \end{aligned}
    \end{equation}
    \ignorespacesafterend
}
\numberwithin{equation}{section}

\usepackage[ style=numeric, giveninits=true, eprint=false, url=false,isbn=false,  language=english, backend=biber
]{biblatex}
\renewbibmacro{in:}{
	\ifentrytype{article}
	{}
	{\printtext{\bibstring{in}\intitlepunct}}}
\AtEveryBibitem{\clearfield{month}}
\AtEveryBibitem{\clearfield{day}}

\usepackage[hidelinks]{hyperref}

\title{node2vec or triangle-biased random walks: stationarity,\\ regularity \& recurrence}
\author[1]{Luca Avena}
\author[1]{Gianmarco Bet}
\author[2]{Lars Schroeder}
\author[2]{Clara Stegehuis}
\affil[1]{Universit\`a degli Studi di Firenze, Italy}
\affil[2]{University of Twente, Netherlands}

\begin{document}

\maketitle

\begin{abstract}
The node2vec random walk is a non-Markovian random walk on the vertex set of a graph, widely used for network embedding and exploration.
This random walk model is defined in terms of three  parameters which control the probability of, respectively, backtracking moves, moves within  triangles, and moves to the remaining neighboring nodes. From a mathematical standpoint, the node2vec random walk is a nontrivial generalization of the non-backtracking random walk and thus belongs to the class of \textit{second-order} Markov chains.
Despite its widespread use in applications, little is known about its long-run behavior. The goal of this paper is to begin exploring its fundamental properties on arbitrary graphs. 
To this aim, we show how lifting the node2vec random walk to the state spaces of \emph{directed edges} and \emph{directed wedges} yields two distinct Markovian representations which are key for its asymptotic analysis.
Using these representations, we find mild sufficient conditions on the underlying finite or infinite graph to guarantee ergodicity, reversibility, recurrence and characterization of the invariant measure.
As we discuss, the behavior of the node2vec random walk is drastically different compared to the non-backtracking random walk. While the latter simplifies on arbitrary graphs when using its natural edge Markovian representation thanks to bistochasticity, the former simplifies on \textit{regular} graphs when using its natural \textit{wedge} Markovian representation. Remarkably, this representation reveals that a graph is regular if and only if a certain \emph{weighted Eulerianity} condition holds.
\end{abstract}

\section{Introduction: random walks, backtracking and local triangles}

Random walks on graphs constitute a fundamental class of stochastic processes at the interface of discrete mathematics and probability theory, providing a versatile framework for modeling network-based systems that evolve through successive random transitions. Arising from elementary considerations of paths and randomness, their study has developed into a deep and far-reaching theory, with strong connections to combinatorics, statistical mechanics, and spectral graph theory. In addition to their intrinsic mathematical interest, random walks underpin a broad spectrum of applications, ranging from the analysis of diffusion phenomena in the natural sciences to the design of efficient algorithms in computer science, including methods for recommender systems~\cite{gori_itemrank_2007}, image clustering~\cite{meila_random_2001}, sampling~\cite{ribeiro_sampling_2012}, and network exploration~\cite{costa_exploring_2007}. Their effectiveness lies in their ability to encode both local transition mechanisms and global structural properties of the underlying graph, thereby offering a unifying perspective for the analysis of complex systems across disciplines. Consequently, over the years a wide variety of different models have been explored within the framework of random walks on graphs. Prominent classical examples include: Markov processes, chains with long-memory, models in static or dynamic random environments, self-avoiding walks, polymer models and loop-erased random walks.

The present paper focuses on the so-called \emph{node2vec random walk}. This process has been introduced in the popular graph embedding algorithm node2vec~\cite{grover_node2vec_2016}, which is used for feature learning in networks with applications to multi-label classification and link prediction. 
The stochastic evolution of the node2vec random walk depends on the previous visited node and on the triangles that contain both the current and the previous node, which we refer to as \textit{local triangles} for simplicity. More specifically, the transitions of the node2vec random walk are determined by three parameters which allow to reward (resp.~penalize) backtracking moves and moves within  local triangles. In particular, such a random walk is not a Markov chain, but a \textit{second-order} Markov chain, see e.g. \cite{FASINO_TONETTO_TUDISCO_2023,HU2014183,Diaconis2013}. 

The study of second-order and, more generally, higher-order Markov chains associated to a graph has attracted significant interest in recent years, both in fundamental and applied literature. In this context, considerable attention has been devoted to the so-called \emph{non-backtracking random walk}~\cite{hermon_reversibility_2019,fitzner_hofstad_nbrw,alon_non-backtracking_2007,benhamou:hal-01141192,glover2026effectsbacktrackingpagerank,Kempton2016NonbacktrackingRW,Sodin2007RandomMN}, and to its variants such as the \emph{$\alpha$-backtracking}~\cite{barot_community_2021,meng_analysis_2020,hermon_reversibility_2019} and the \emph{$k$-non-backtracking walk}~\cite{hermon_reversibility_2019}. In fact, the node2vec random walk is a parametric model that generalizes both the non-backtracking and the $\alpha$-backtracking random walks. Yet, the mathematical structure of the latter models are significantly simpler than the more general node2vec. Indeed, after a standard lifting of these simpler (though still challenging) second-order chains from the vertex space to the space of directed edges, one obtains a Markovian representation whose associated transition matrix is doubly stochastic. Thanks to this, the invariant measures of the non-backtracking and the $\alpha$-backtracking random walks are obtained by solving appropriate counting problems in the directed edge space.

The main goal of this article is therefore to explore the basic mathematical structure and the first natural questions of the general node2vec random walk on arbitrary finite and infinite graphs.
As we shall discuss, the node2vec random walk is only bistochastic on highly artificial graphs or in very specific parameter regimes. For this reason, despite its popularity in the computer science community, the mathematical literature on the node2vec random walk is still limited and is typically confined to parameter regimes in which the key difficulty, that is distinguishing between local triangle moves and other moves, is absent~\cite{barot_community_2021,meng_analysis_2020}. To the best of our knowledge, the recent work \cite{schroeder_stationary_2025} contains the only mathematical results concerning node2vec random walks in the most general parameter regime. Therein, the authors manage to explicitly characterize the invariant measure of the node2vec random walk with general parameters on a certain class of finite graphs, known as \emph{household models}. The structure of the graphs in this class exhibits a topological contraction procedure which allows one to couple the node2vec random walk to an $\alpha$-backtracking random walk. As a consequence,  the characterization of the invariant measure is essentially reduced to a counting problem.

\paragraph{Main contributions at glance and paper organization}
We start in Section \ref{sec:model} by giving the precise model definition and introducing the two Markovian representations for the spaces of directed edges and wedges, respectively.
We state the three main results in Section \ref{sec:results}. Theorem \ref{thm:ergodicity} is referred to as the node2vec ergodic theorem, in analogy with the ergodic theorem for Markov chains on a countable state space. The result distinguishes between the case where the backtracking parameter $\alpha$ is zero and the case where it is not, providing mild sufficient conditions for both cases under which the process forgets its initial position. Under these conditions, the process stabilizes in the long run to a unique invariant measure. The delicate structure of this measure simplifies in wedge space; see \eqref{eq:simplified invariant equation wedges}.

Theorem \ref{thm:statdistonregulargraph} shows that the complicated structure of the node2vec random walk gets drastically simplified when considering $d$-regular graphs. In fact, only such graphs satisfy a certain \emph{Eulerianity condition}, see \eqref{eq:inandoutwedgesequal}, allowing to characterize the invariant measure of node2vec random walks in terms of local wedge weights, see \eqref{eq:stat dist reg graph}. As it turns out, the above mentioned Eulerianity can not be easily interpreted in terms of classical electrical network theory,
as such a theory is not yet fully understood in directed spaces. Yet, such a Eulerianity implies \emph{directed detailed balance conditions}, see Definition~\ref{def:EDB and WDB}, which characterize reversibility of the process in the directed edge and wedge space representations. Interestingly, the mentioned Eulerianity is not equivalent to such directed detailed balance. 

The third principal result, Theorem \ref{thm:recurrence}, addresses recurrence on infinite graphs under distinct, mild ergodic conditions, depending on whether backtracking moves are permitted. It states that, provided the directed detailed balanced condition is satisfied (e.g., if the graph is $d$-regular), then the node2vec random walk is recurrent on a graph $G$ if and only if the simple random walk is recurrent on $G$. The proof adapts the approach of \cite{hermon_reversibility_2019}, which establishes recurrence comparisons for the non-backtracking and the $k$-non-backtracking random walk, to the node2vec random walk. As an immediate corollary of Theorem \ref{thm:recurrence}, we obtain that node2vec on triangular lattices is recurrent. 
Section \ref{sec:Examples} is devoted to discussing examples, including subtle statements and counterexamples. In particular, we clarify why the analysis of node2vec with $\beta\neq\gamma\neq0$ cannot, in general, be reduced to a counting problem, except on highly artificial graphs that enforce bistochasticity. We also provide a counterexample demonstrating that directed detailed balance is not equivalent to regularity.

The proofs of the main three theorems are presented in the remaining sections. In Section \ref{ErgRegProofs} we present the proof of Theorem \ref{thm:ergodicity} and Theorem \ref{thm:statdistonregulargraph}, and in Section \ref{RecProofs} we present the proof of the recurrence. 

\section{The model and its Markovian representations}\label{sec:model}
For $n \in \mathbb{N}$, we write $[n] \coloneqq \{1,\ldots,n\}$.
Throughout this article, we assume that $G = (V,\bE)$ is a simple (i.e., without self-loops and multi-edges) undirected graph with node set $V$ and edge set $\bE$. We define the \textit{directed edge set} $\vE$ as
\begin{align} 
    \vE \coloneqq \{ (v_1,v_2) : v_1,v_2 \in V,\ \{v_1,v_2\} \in \bE \}. 
\end{align}
We denote the degree of a node $v \in V$ by $d_v$. We are interested in the following random walk on the graph $G$.

\begin{definition}\emph{(node2vec random walk) }
\label{node2vecDef} Given an undirected graph $G=(V,\bE)$ and parameters $\alpha,\beta,\gamma \ge 0$, a discrete-time stochastic process $X = (X_n)_{n \geq 0}$ with values in $V$ and path measure $\mathbb{P}$ is called \emph{node2vec random walk} if, for any $s,u,v\in V$ such that $\{s,u\},\{u,v\} \in \bE$, and for any $n\geq 1$,

\begin{align}\label{eq:transprobn2v}
\mathbb{P}(X_{n+1} = v \mid X_n = u,\, X_{n-1} = s) \propto 
\begin{cases} 
\alpha, & \text{if } s = v,\\
\beta, & \text{if } \{s,v\} \in \bE,\\ 
\gamma, & \text{if } \{s,v\} \notin \bE \text{ and } s \neq v.
\end{cases} 
\end{align}
\end{definition}
We denote the initial distribution of the node2vec random walk for a directed edge $(u_0,v_0)\in \vE$ by 
\begin{align}
    p_{*}((u_0,v_0))\coloneqq\mathbb P(X_0=u_0, X_1=v_0).
\end{align}
The random walk \eqref{eq:transprobn2v} is a parametric \emph{second-order Markov chain}, see e.g. \cite{FASINO_TONETTO_TUDISCO_2023}. The parameters $\alpha,\beta,\gamma$ control the probability of, respectively, the backtracking steps, the steps within local triangles and the remaining steps, which we call \textit{forward} steps. We give a pictorial representation of the role of the parameters in Figure~\ref{fig:node2vecdef}. Special parameter cases include the \textit{simple random walk} ($\alpha = \beta = \gamma$), the \textit{non-backtracking random walk} ($\alpha =0, \beta = \gamma$) and the \textit{$\alpha$-backtracking random walk} ($\alpha > 0, \beta = \gamma$). 
As the main novelty of this article is to analyze the random walk in full generality in a non-trivial regime, i.e., $\beta \neq \gamma$, we assume $\beta,\gamma>0$. Still, we will consider and distinguish the two cases $\alpha > 0$ and $\alpha = 0$, since allowing backtracking drastically changes the underlying ergodic structure.
\begin{figure}[tbp]
    \centering
    \includegraphics[width=0.3\linewidth]{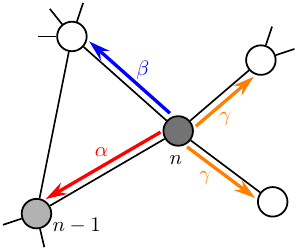}
    \caption{Illustration of the transition rates of a node2vec random walk.}
    \alt{Illustration of the transition rates of a node2vec random walk.}
    \label{fig:node2vecdef}
\end{figure}
While $X$ is not a Markov chain on the state space $V$, by lifting it to the higher-order spaces of the directed edges and the directed wedges of $G$, we get two different 
Markovian representations which will be at the core of our analysis. Next, we formally introduce the notions required by our analysis.

For $e = (v_1,v_2) \in \vE$, we write $e(1) = v_1$ for the starting node and $e(2) = v_2$ for the ending node and by $-e = (v_2,v_1)$ we denote the directed edge pointed in the opposite direction of $e$. A \emph{directed wedge} is a length‑two directed walk in $G$. Thus, the \emph{directed wedge set} $W$ in $G$ is 
\begin{align}\label{eq:wedge space definition}
    W \coloneqq \Big\{ w = (w(1),w(2),w(3)) = (v_1,v_2,v_3) : v_1,v_2,v_3 \in V,\ (v_1,v_2) \in \vE,\ (v_2,v_3) \in \vE \Big\}. 
\end{align} 
For each wedge $w \in W$, define its first and second directed edges as 
\begin{align} 
    e_1(w) \coloneqq (w(1),w(2)), \qquad e_2(w) \coloneqq (w(2),w(3)), 
\end{align} 
and its inverse wedge as 
\begin{align} 
    -w \coloneqq (w(3),w(2),w(1)), 
\end{align}
which corresponds to traversing $w$ backwards.

The transition probabilities of the node2vec random walk~\eqref{eq:transprobn2v} depend on whether the next step is a backtracking step, a local triangle step or a forward step. In other words, the transition probabilities depend on the \textit{last two} visited nodes. Consequently, instead of analysing the process on $V$, we consider its Markovian representation on the space of directed edges $\vE$. Although the resulting walk on $\vE$ is Markovian, it can be more insightful to work with a representation for which the state space can be partitioned according to the three different transition types of~\eqref{eq:transprobn2v}. This is accomplished by lifting the node2vec random walk to the space of directed wedges $W$, so that each wedge-state encodes exactly the information required to determine the transition weights in~\eqref{eq:transprobn2v}. For a related notion of lifting, see~\cite{chen_lifting_1999}. Next, we formally define the two Markovian representations that we consider. 

\paragraph{Edge Markovian representation}
Define the lifted graph $\vv{G} = (\vE,W)$, where $\vE$ and $W$ are the set of directed edges and wedges defined above. Let  $\vv{Y} = (\vv{Y}_n)_{n \ge 1}$ be the (time-homogeneous) Markov chain on $\vE$ with transition kernel $\vv{P}$ given by 
\begin{align} 
    \vv{P}(e,e') \coloneqq \mathbb{P}(X_3 = e'(2) \mid X_2 = e'(1),\, X_1 = e(1))\,  
\end{align}
for any $e,e'\in \vE$ such that $e(2)=e'(1)$, and $\vv{P}(e,e')=0$ otherwise. Thus, $\vv{Y}$ can move from $e$ to $e'$ exactly when the ending node of $e$ matches the starting node of $e'$, i.e., when $(e,e')\in W$. $\vv{Y}$ is initialized by the distribution $p_*$ defined above. For the analysis of this chain, it is convenient to define the \textit{edge in-neighbors} of a node $v\in V$ as
\begin{align}\label{eq:edge in neighbors of v}
    \mathrm{IN}_{\vE}(v) &= \{ e'\in \vE: e'(2) = v \}.
\end{align}

\paragraph{Wedge Markovian representation}
We begin by introducing a useful partition of the set of wedges 
\begin{align}\label{eq:wedge space decomposition}
    W = W_{\Delta} \cup W_{\Lambda} \cup W_{-},
\end{align}
where 
\begin{align}
	W_{\Delta} \coloneqq \{ w\in W: \{w(3),w(1)\} \in \bE \}
\end{align}
is the subset of \textit{triangle wedges},
\begin{align}
	W_{\Lambda} \coloneqq \{ w\in W: \{w(3),w(1)\} \notin \bE, w(1) \neq w(3) \}
\end{align}
is the subset of \textit{open wedges}, and 
\begin{align}
    W_{-} \coloneqq \{ w\in W: w(1) = w(3)  \}
\end{align}
is subset of \textit{flat wedges}.
By construction, $W_{\Delta}, W_{\Lambda}$, and $W_{-}$ are pairwise disjoint. For any wedge $w\in W$, its \textit{type} $\lambda (w)$ is
\begin{align}
	\lambda (w) \coloneqq 
    \begin{cases}
        \alpha, \text{ if } w \in W_{-}, \\
		\beta, \text{ if } w \in W_{\Delta}, \\
		\gamma, \text{ if } w \in W_{\Lambda}.
    \end{cases}
\end{align}
As $\lambda$ is not depending on the direction of the wedge, for any $w\in W$ it holds
\begin{align}\label{eq:wedgefact1}
    \lambda(w) = \lambda(-w).
\end{align}
Similar to~\eqref{eq:edge in neighbors of v}, we introduce the \textit{wedge in-neighbors} of a directed wedge $w\in W$ as
\begin{align}
	\mathrm{IN}_W(w) \coloneqq \{ w'\in W: w'(2) = w(1), w'(3) = w(2), \{w'(1),w(1)\} \in \bE  \},
\end{align}
which corresponds to all wedges $w'$ which concatenate into $w$, see Figure \ref{fig:in-and-outneighbors}. In other words, $\mathrm{IN}_W(w)$ contains all the wedges whose second edge equals the first edge of $w$.
With an abuse of notation, we denote the wedge in-neighbors of a directed edge $e\in \vE$ and of a node $v\in V$ as
\begin{align}
    \mathrm{IN}_W(e) &= \{ w'\in W: w'(2) = e(1), w'(3) = e(2)  \}, \\
    \mathrm{IN}_W(v) &= \{ w'\in W: w'(3) = v \}.
\end{align}%
These are all the wedges that end in $e$ (respectively, end in $v$). To simplify the notation, from now on we will drop the index $W$ or $\vE$ from the in-neighbor sets since it will be clear from the context. 

Finally, for each wedge $w\in W$ we define its out-neighbor set as all wedges that share the first two nodes with $w$:
\begin{align}
	\mathrm{ON}(w) = \{ w'\in W: w'(1) = w(1), w'(2) = w(2), \{w(2),w'(3)\} \in \bE  \},
\end{align}
representing all possible next steps from the directed edge $e_1(w)$, see Figure~\ref{fig:in-and-outneighbors}. \\
\begin{figure}[tbp]\label{fig:in-and-outneighbors}
    \centering
    \begin{minipage}{0.5\textwidth}
    \includegraphics[width=0.9\linewidth]{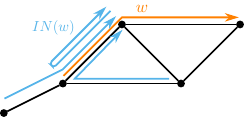}
    \end{minipage}\hfill
    \begin{minipage}{0.5\textwidth}
            \centering
            \includegraphics[width=0.9\linewidth]{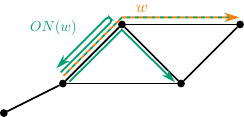}
    \end{minipage}
    \caption{Illustration of the in-and out-neighbors of a fixed wedge $w$}
    \alt{Illustration of the in-and out-neighbors of a fixed wedge.}
\end{figure}
With the introduced notation we can phrase the transitions of the node2vec random walk on the wedge space. The probability of choosing a specific next wedge $w$ from all possible transition options that are given by $\mathrm{ON}(w)$ is 
\begin{align}\label{eq:defwedgetransition} 
    p(w) = \frac{\lambda(w)}{\sum_{w' \in \mathrm{ON}(w)} \lambda(w')}. 
\end{align}
We define the wedge chain to be the Markov chain $\hY =(\hY_n)_{n\ge 1}$ with state space $W$, and $w,w' \in W$ and transition matrix
\begin{nalign}\label{eq:wedge transition definition}
    \hP(w,w') &\coloneqq 
     \vv{P}(e_3,e_4)\indic_{ \{ e_2 = e_3 \} } = \mathbb{P}(\vv{Y}_3 = e_4|\vv{Y}_2 = e_3 = e_2,\vv{Y}_1 = e_1) \\
     %
    &= \PP(X_4 = w'(3)| X_3 = w'(2), X_2 = w'(1))\indic_{ \{ w \in \mathrm{IN}(w') \}} \\
    &= p(w') \indic_{ \{ w \in \mathrm{IN}(w') \}},
\end{nalign}
for directed wedges $w= (e_1,e_2)$ and $w'= (e_3,e_4)$.
Thus, $w$ can move to $w'$ exactly when the second edge of $w$ equals the first edge of $w'$ and the transition probability is proportional to the wedge type  $\lambda(w')$. For all $w\in \mathrm{IN}(w')$ we get $\hP(w,w') = p(w')$ which means that the rates do not depend on $w$. Note that any initial condition $\hat{p}_{*}$ for $\hY$ must satisfy 
\begin{nalign}
    \hat{p}_{*}(w) \coloneqq ~&\PP(X_0 = w(1), X_1 = w(2), X_2 = w(3)) \\
    = ~&\PP(X_0 = w(1), X_1 = w(2))\PP(X_2 = w(3) \mid X_0 = w(1), X_1 = w(2)).
\end{nalign}
Hence, $\hat{p}_{*}$ is uniquely determined by the choice of the initial condition $p_{*}$ for the edge chain $\vv{Y}$. 

The $n$-step transition kernel of $\hY$ is given by
\begin{align}\label{eq:nsteptransition}
    \hP^n(w,w') = \PP(X_{n+2} = w'(3), X_{n+1} = w'(2), X_n = w'(1)| X_2 = w(3),X_1 = w(2)).
\end{align}
The Markov chain on the wedge space can be interpreted as a double lifting of the original process. Indeed, for the first lifting the state space of $\vv{Y}$ is the edge set $\vE$ of the original process $X$, and the state space of $\hY$ is the edge space $W$ of the lifted chain $\vv{Y}$.
Note that $\hP$ is not the same as the process with kernel $\hP'\coloneqq  \vv{P}^2$ on $W$ that takes two consecutive steps on the directed edge set as the latter is sub-Markov on $W$.

\section{Main results}\label{sec:results}

Our first main result is concerned with the ergodicity of the node2vec random walk on a graph $G$. In it, we collect the necessary conditions on $G$ under which the transition kernel $\hP$ is irreducible and aperiodic. We further show that $\hY$ has a simpler invariant equation. Finally, we express the stationary distribution of the node2vec random walk on $V$ in terms of the stationary distribution of the lifted chain $\hY$ on $W$ via a pullback.
\begin{theorem}[node2vec ergodic theorem]\label{thm:ergodicity}
Let $G$ be a connected graph, and let $\alpha\ge 0$ and $\beta,\gamma>0$.

\begin{enumerate}
    \item[$(i)$] If $\alpha>0$, then the transition kernel $\hP$ is irreducible. Moreover, if $G$ contains at least one triangle, then $\hP$ is also aperiodic.
    
    \item[$(ii)$] If $\alpha=0$ and $\min_{v\in V} d_v \ge 2$ and $\max_{v\in V} d_v > 2$,
    then $\hP$ is irreducible.
\end{enumerate}

Furthermore, when $G$ is finite and $\hP$ is irreducible and aperiodic, $\hY$ admits a unique invariant probability measure $\hpi$, which satisfies
\begin{equation}\label{eq:simplified invariant equation wedges}
    \hpi(w)
    = p(w)\sum_{w'\in \mathrm{{IN}}(w)} \hpi(w'),
    \qquad w\in W.
\end{equation}
As a consequence, $X$ possesses a unique limiting measure $\pi$ on $V$, given by
\begin{equation}\label{eq:pullback wedges to nodes}
    \pi(v)
    = \lim_{n\to\infty} \PP\bigl(X_n=v \,\big|\, X_2=u_2,\,X_1=u_1\bigr)
    = \sum_{w\in \mathrm{{IN}}(v)} \hpi(w),
\end{equation}
for all $v\in V$ and all $\{u_1,u_2\}\in \vE$.
\end{theorem}

\begin{remark}[Positive or zero back-tracking parameter $\alpha$]
    The case $\alpha = 0$ behaves differently from the case $\alpha> 0$. When $\alpha=0$, additional structural assumptions on the graph $G$ are required to ensure irreducibility of the transition kernel $\hP$. 
Specifically, the condition $\min_{v\in V} d_v \ge 2$ prevents the walk from becoming trapped at vertices of degree one, while $\max_{v\in V} d_v > 2$ excludes the degenerate case in which $G$ is a simple cycle. Together, these conditions guarantee irreducibility of $\hP$.
On the other hand, the aperiodicity of $\hP$ for $\alpha=0$ cannot be characterized by a simple local condition. 
Indeed, $\hP$ is aperiodic if and only if $G$ contains at least two cycles whose lengths have greatest common divisor equal to one. 
This is precisely the definition of an aperiodic graph, and therefore no simpler necessary condition exists in this setting.
\end{remark}

The assumption that $G$ contains at least one triangle is quite weak, and arises from the distinction between triangles and non-triangles in node2vec. In fact, if $G$ does not contain any triangle, then the transition probabilities of node2vec do not depend on $\beta$, and the process is essentially an $\alpha$-backtracking random walk. 
\begin{remark}[Edge and wedge representation]\label{rem:thm1 also holds for edges}
Theorem~\ref{thm:ergodicity} can equivalently be formulated using the edge Markovian representation. 
More precisely, by replacing the wedge-level quantities $\hP$, $\hY$, $\hpi$ and $w,w'\in W$ with their edge-level counterparts $\vP$, $\vv{Y}$, $\vpi$ and $e,e'\in \vE$, all statements of the theorem remain valid.
The only exception is the simplified invariant relation \eqref{eq:simplified invariant equation wedges}, which relies on the wedge structure and does not admit an analogue on the edge space $\vE$.
\end{remark}
The following lemma clarifies the precise relationship between the invariant measures of the edge and wedge Markovian representations.
\begin{lemma}[Relation of invariant measures on edge and wedge spaces]\label{lem:relation inv measures on E and W}
    Let $G$ be connected. Assume that $\vP$ and $\hP$ are irreducible and aperiodic. Then, there exist unique invariant vectors $\vpi$ for $\vv{Y}$ and $\hpi$ for $\hY$ that satisfy
    \begin{align}\label{eq:pullback wedges to edges}
        \vpi(e) = \sum\limits_{w\in \mathrm{IN}(e)} \hpi(w) 
    \end{align}
    for all $e\in \vE$ and moreover,
    \begin{align}\label{eq:relation edges to wedges}
        \hpi(w) = \vpi(e_1(w))\vP(e_1(w),e_2(w))
    \end{align}
    for all $w\in W$.
\end{lemma}

We now introduce a notion of directed detailed balance for the edge and wedge Markovian representations of node2vec. This condition captures a natural form of reversibility that respects the directed structure of the lifted states and will be central to our following results on the stationary distribution and recurrence.

\begin{definition}[Edge and wedge directed detailed balance]\label{def:EDB and WDB}
    We say that the vector $\vpi$ fulfills the \emph{edge directed detailed balance} if
    \begin{align}\label{eq:antireversibilityedges}
        \vpi(e)\vP(e,e') = \vpi(e')\vP(-e',-e)\tag{EDB}
    \end{align}
    for all $e,e'\in \vE$.\\
    We say that the vector $\hpi$ fulfills the \emph{wedge directed detailed balance} if
    \begin{align}\label{eq:antireversibility}
        \hpi(w)\hP(w,w') = \hpi(w')\hP(-w',-w)\tag{WDB}
    \end{align}
    for all $w,w'\in W$.
\end{definition}
Our next theorem shows that regular graphs are somewhat special for node2vec random walks. Indeed, we show that a graph is regular if and only if it satisfies a certain weighted Eulerianity condition. As a consequence,
in the wedge space the stationary measure takes a remarkably simple closed form. Furthermore, node2vec random walks are reversible on regular graphs, in the sense expressed in~\eqref{eq:antireversibility}. The latter reversibility is not equivalent to regularity of the graph as we shall discuss below.
\begin{theorem}[node2vec on regular graphs]\label{thm:statdistonregulargraph}
    Let $G$ be finite and connected. Then, $G$ satisfies the weighted Eulerianity condition
    \begin{align}\label{eq:inandoutwedgesequal}
        \sum\limits_{w'\in \mathrm{ON}(w)}\lambda(w') = \sum\limits_{w'\in \mathrm{IN}(w)}\lambda(w')
    \end{align}
    for all $w\in W$ if and only if $G$ is a regular graph. \\
    Let $\hP$ be irreducible and assume that $G$ satisfies~\eqref{eq:inandoutwedgesequal}. Define $Z \coloneqq \sum_{w'\in W}\lambda(w')$. Then, we can express the stationary distributions of $\hY$ and $X$, respectively on $W$ and $V$ as
    \begin{align}\label{eq:stat dist reg graph}
        \hpi(w) = \frac{\lambda(w)}{Z},\qquad \pi(v)= \sum\limits_{w\in \mathrm{IN}(v)} \frac{\lambda(w)}{Z}
    \end{align}
    for all $w\in W$ and $v\in V$. Additionally, $\hpi$ satisfies~\eqref{eq:antireversibility}.
\end{theorem}
We call the condition~\eqref{eq:inandoutwedgesequal} the weighted Eulerianity condition because it can be understood as a weighted extension of Eulerianity on directed graphs. Indeed, a directed graph is called Eulerian if all vertices have equal in- and out-degree. In particular, this is the one case where the simple random walk on a directed graph exhibits an explicit expression of its stationary distribution, see e.g.~\cite{aldousfill}.

For a regular graph, the stationary distribution of node2vec looks quite different from the degree-proportional stationary distribution on $V$ of the simple random walk, non-backtracking random walk, or $\alpha$-backtracking random walk.  In particular, the stationary distribution $\pi$ obtained in \eqref{eq:stat dist reg graph} is a normalized linear combination of the parameters $\alpha$, $\beta$, and $\gamma$, weighted by the number of incoming wedges of the corresponding type.

In a $d$-regular graph, each node $v\in V$ has exactly $d$ incoming flat wedges, two incoming triangle wedges for every triangle containing $v$ and the remaining incoming wedges are open wedges. Altogether, the total number of incoming wedges satisfies $\lvert\mathrm{IN}_W(v)\rvert =d^2$. Consequently, when $\beta$ is large, nodes participating in many triangles receive higher stationary mass, whereas a large value of $\gamma$ favors nodes with many incoming open wedges. Since the number of incoming flat wedges is constant across all nodes, the influence of $\alpha$ on $\pi(v)$ is determined by its relative magnitude compared to $\beta$ and $\gamma$, as well as by the number of triangles incident to $v$.

The next corollary shows that for a finite regular graph, also edge directed detailed balance holds.
\begin{corollary}\label{cor:anti-rev on edges}
    If $G$ is a finite regular graph, then~\eqref{eq:antireversibilityedges} holds for a vector $\vpi$. Moreover, 
    \begin{align}\label{eq:regular opposite edges equal stat}
        \vpi(e) = \vpi(-e)
    \end{align}
    for all $e \in \vE$.
\end{corollary}
In the next lemma we investigate the relationship between the edge detailed balance and the wedge detailed balance conditions for the node2vec random walk on an arbitrary graph.
\begin{lemma}[Relation of edge and wedge directed detailed balance]\label{lem:directed detailed balance relations}
    Let $G$ be connected. If a vector $\vpi$ satisfies~\eqref{eq:antireversibilityedges}, then
    \begin{itemize}
        \item[$(1a)$] $\vpi$ is invariant, i.e., $\vpi\vP = \vpi$,
        \item[$(1b)$] if $\alpha >0$, $\vpi(e) = \vpi(-e)$ for all $e\in \vE$,
        \item[$(1c)$] if $\alpha =0$, then 
        $\vpi(e) = \vpi(-e)$ for all $e\in \vE$ if and only if
        \begin{align}\label{eq:eulerian cycle condition}
            \prod\limits_{i=1}^{n} \sum\limits_{w\in \mathrm{ON}((e_i,\cdot))}\lambda(w) = \prod\limits_{i=1}^{n} \sum\limits_{w\in \mathrm{IN}((e_i,\cdot))}\lambda(w)
        \end{align}
        holds for all cycles $e_1,\ldots,e_{n}$ with $n> 2$ such that $e_i \in \mathrm{IN}(e_{i+1})$ for all $i\in [n-1]$ and $e_{n} \in \mathrm{IN}(e_1)$,
        \item[$(1d)$] if $\vpi(e) = \vpi(-e)$ for all $e\in \vE$, \eqref{eq:antireversibility} holds for some vector $\hpi$ for which $\hpi(w)=\hpi(-w)$ for all $w\in W$.
    \end{itemize}
    If a vector $\hpi$ satisfies~\eqref{eq:antireversibility}, then
    \begin{itemize}
        \item[$(2a)$] $\hpi$ is invariant, i.e., $\hpi\hP = \hpi$,
        \item[$(2b)$] if $\alpha >0$, then $\hpi(w) = \hpi(-w)$ for all $w\in W$.
    \end{itemize}
\end{lemma}
This characterizes the precise relationship between directed detailed balance, invariance, and symmetry properties of stationary measures for node2vec. In particular, it shows that directed detailed balance is a strong condition: whenever it holds, invariance follows automatically, and for $\alpha>0$ it enforces symmetry under direction reversal at both the edge and wedge levels. 

\begin{remark}
    If $G$ is a finite regular graph, then $\hpi(w)=\hpi(-w)$ is trivially fulfilled, since $\hpi(w) = \frac{\lambda(w)}{Z}$ and $\lambda(w)=\lambda(-w)$. However, the reverse implication does not hold and wedge directed detailed balance does not imply that $G$ is regular, as we will see in Example~\ref{ex:clique with missing edge}.
\end{remark}
We next show that, on regular graphs, the global recurrence behavior of node2vec coincides with that of the simple random walk (SRW). 
We recall that $X$ on $G$ is called \emph{recurrent} if every node $v\in V$ is visited infinitely often almost surely.

For the recurrence result, we require the following geometric condition on the graph: there exists $R>0$ such that every ball of radius $R$ contains at least one cycle, i.e.,
\begin{equation}\label{eq:condition every ball contains cycle}
    \exists R>0:~\forall x\in V \ \exists \text{ cycle } C \subset V \text{ such that } C \subset B_R(x) \coloneqq \{y\in V \mid d_G(x,y)\le R\}.
\end{equation}

\begin{theorem}[node2vec recurrence criterion]\label{thm:recurrence}
    Let $G$ be an infinite connected graph with bounded degree. Let $X$ be a node2vec random walk with parameters $\alpha \ge 0$ and $\beta, \gamma >0$ on $G$. Assume that~\eqref{eq:antireversibilityedges} is fulfilled for a vector $\vpi$. If $\alpha = 0$, additionally assume $d_v \ge 2$ for all $v\in V$, $\vpi(e) = \vpi(-e)$ for all $e\in \vE$ and that there exists $R>0$ such that the cycle condition \eqref{eq:condition every ball contains cycle} is fulfilled. \\
    Then, $X$ is recurrent on $G$ if and only if the SRW is recurrent on $G$.
\end{theorem}

In the proof of Theorem \ref{thm:recurrence}, we use the relation~\eqref{eq:antireversibilityedges} to construct a reversible Markov chain $Q$ on the undirected edges $\bE$ that is a subsampling of a lazy version of $\vv{Y}$, only making jumps in case the lazy version of $\vv{Y}$ attempts to backtrack. Since $Q$ was built from $\vv{Y}$, it can be shown that one is recurrent if and only if the other is recurrent. By using the reversibility of $Q$, we can apply electrical network theory to compare it to the SRW to conclude the proof.
The proofs for $\alpha >0$ and $\alpha = 0$ are quite different because the construction of $Q$ depends on the returns from a directed edge $e$ to its reversal $-e$. The probability $\vP(e,-e)$ is only positive for $\alpha >0$, which means that the case $\alpha = 0$ needs a more elaborate construction.

Since the SRW is recurrent on the infinite triangular lattice, we obtain the following corollary from the previous theorem.
\begin{corollary}[Recurrence of node2vec on infinite triangular lattice]
    Let $G$ be the infinite triangular lattice and let $X$ be a node2vec random walk with $\alpha \ge 0$ and $\beta,\gamma >0$. Then, $X$ is recurrent on $G$.
\end{corollary}
\section{Discussion and Examples}\label{sec:Examples}

In this section, we discuss some examples to illustrate the behavior of the node2vec random walk and we provide a counterexample for the relationship between the weighted Eulerianity~\eqref{eq:inandoutwedgesequal} and the directed detailed balance conditions. Example~\ref{ex:bistoch of alpha-brw}, Example~\ref{ex:bistoch on E} and Lemma~\ref{lem:bistoch on W} demonstrate that the analysis of the invariant measure of the node2vec random walk cannot be reduced to a counting problem when $\beta\neq\gamma$, in contrast with the non-backtracking and $\alpha$-backtracking cases.
Concretely, we investigate under which conditions the kernels $\vP$ and $\hP$ are bistochastic and what that means for their stationary distribution. For that, we assume throughout all three examples that $G$ is connected and finite and that $\vP$ and $\hP$ are irreducible and aperiodic. Following our conventions, we denote the stationary distribution of the node2vec random walk lifting on $\vE$ by $\vpi$ and on $W$ by $\hpi$. As a guiding motivation, in our first example we recall a well-known result on the $\alpha$-backtracking random walk~\cite{meng_analysis_2020}.

\begin{example}[Bistochasticity of $\alpha$-backtracking random walk]\label{ex:bistoch of alpha-brw}
Consider the special node2vec random walk with parameters $\alpha > 0, \beta = \gamma$, called $\alpha$-backtracking random walk, whose transition matrix on $\vE$ is denoted by $\vP_\alpha$. Note that for $\alpha = \beta = \gamma$, we retrieve the SRW and for $\alpha = 0$, we retrieve the non-backtracking random walk. Checking the column sums of the $\alpha$-backtracking random walk
\begin{align}
    \sum\limits_{e\in \vE}\vP_\alpha(e,e') = \frac{\alpha}{\alpha + (d_{e'(1)}-1)\beta} + \frac{(d_{e'(1)}-1)\beta}{\alpha + (d_{e'(1)}-1)\beta} = 1.
\end{align}
for all $e,e'\in \vE$, we see that $\vP_\alpha$ is bistochastic on $\vE$.
Now, if we consider the invariant equation $\vpi = \vpi \vP_\alpha$ and choose $\vpi = \tfrac{1}{|\vE|}$ as the uniform measure on $\vE$, we obtain
\begin{nalign}\label{eq:bistochastic equivalent to uniform measure}
   \frac{1}{|\vE|} = \sum\limits_{e\in \vE}\frac{1}{|\vE|} \vP_\alpha(e,e') = \frac{1}{|\vE|}\sum\limits_{e\in \vE} \vP_\alpha(e,e')
\Longleftrightarrow \sum\limits_{e\in \vE} \vP_\alpha(e,e') = 1
\end{nalign}
which means that $\vpi$ is the uniform distribution if and only if $\vP_\alpha$ is bistochastic (the same argument also holds for general Markov chains). We conclude that the $\alpha$-backtracking random walk has stationary distribution $\vpi(e) = \tfrac{1}{|\vE|}$ for all $e\in \vE$. Setting $\alpha = 0$ in the above calculations leads to the same stationary distribution for the non-backtracking random walk. To get the stationary distribution on the state space of nodes $V$, we apply the pullback $\pi(v) = \sum_{e\in \mathrm{IN}(v)} \vpi(e) = \tfrac{d_v}{2|\vE|}$ which matches the stationary distributions of the SRW. 
\end{example}
This leads to the natural questions: For which other graphs and parameter choices is $\vpi$ uniform? Moreover, when is $\hpi$ uniform? Regarding the first question, we consider the following example:
\begin{figure}[tbp]
    \centering
    \includegraphics[width=0.2\linewidth]{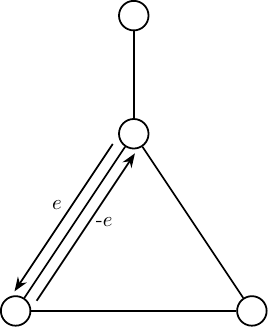}
    \caption{Example graph for studying bistochasticity on the directed edges}
    \alt{Illustration of a four-node graph that is a triangle with one outside arm.}
    \label{fig:bistochasticity example}
\end{figure}
\begin{example}[Bistochasticity of $\vP$]\label{ex:bistoch on E}
Let $\vP$ denote the transition matrix of the edge Markovian representation of the node2vec random walk with parameters $\alpha,\beta,\gamma>0$. To determine when the invariant measure $\vpi$ is uniform, we examine the bistochasticity of $\vP$.

Consider the graph $G'$ shown in Figure~\ref{fig:bistochasticity example}, which consists of a triangle with a single edge attached to it. Following the criterion in \eqref{eq:bistochastic equivalent to uniform measure}, we compute, for a fixed edge $e'\in \vE$, the column sum
\begin{nalign}
    \sum_{e\in \vE} \vP(e,e')
    &= \vP(-e',e') + \vP(e_1,e') + \vP(e_2,e') \\
    &= \frac{\alpha}{\alpha+\beta+\gamma}
       + \frac{\gamma}{\alpha+2\gamma}
       + \frac{\beta}{\alpha+\beta+\gamma},
\end{nalign}
where $e_1$ and $e_2$ denote the two incoming edges of the triangle distinct from $-e'$.

For $\beta,\gamma>0$, this sum equals $1$ if and only if $\beta=\gamma$, corresponding to the $\alpha$-backtracking random walk. Hence, $\vP$ is bistochastic, and therefore admits the uniform distribution as a stationary measure, only in this case.

Finally, the same argument applies to any graph that contains $G'$ as an induced subgraph, a property satisfied by most graphs arising in practical applications. For the case $\alpha = 0$ a similar graph can be constructed.
\end{example}

As seen in the previous example, if we consider graphs that contain a triangle, the stationary distribution on directed edges is uniform only if $\beta = \gamma$, in other words if there is no distinction between forward moves and triangle moves. The following lemma investigates when the stationary distribution $\hpi$ on $W$ is uniform.
\begin{lemma}[Bistochasticity of $\hP$]\label{lem:bistoch on W}
Let $\hP$ be the transition matrix of the wedge Markovian representation of the node2vec random walk with $\alpha \ge 0$ and $\beta, \gamma > 0$. Then, $\hpi$ is the uniform distribution on $W$ if and only if
\begin{align}\label{eq:bistoch of P hat}
    p(w) = |\mathrm{IN}(w)|^{-1}.
\end{align}
\end{lemma}
\begin{proof}
    As seen in Example~\ref{ex:bistoch of alpha-brw}, a stationary distribution is uniform if and only if the corresponding transition matrix is bistochastic. Let $w\in W$.
    Since for all $w' \in \mathrm{IN}(w)$ we have $\hP(w',w) = p(w)$, we obtain
    \begin{align}
        \sum\limits_{w'\in W} \hP(w',w) =  \sum\limits_{w'\in \mathrm{IN}(w)} \hP(w',w) = \sum\limits_{w'\in \mathrm{IN}(w)} p(w) = |\mathrm{IN}(w)|p(w).
    \end{align}
    That implies that $\hP$ is bistochastic if and only if $ p(w) = |\mathrm{IN}(w)|^{-1}$ for all $w \in W$ which gives the wanted result.
\end{proof}
The condition~\eqref{eq:bistoch of P hat} is quite restrictive since $|\mathrm{IN}(w)|\in\mathbb N$. 
To emphasize that it is unrealistic to have a uniform $\hpi$, we compute the projected stationary distribution $\pi$ on $V$ in the case of a uniform $\hpi$. Applying the pullback to the nodes~\eqref{eq:pullback wedges to nodes} yields
\begin{align}
     \pi(v)  = \sum\limits_{w\in \mathrm{IN}(v)} \frac{1}{|W|} = \frac{1}{|W|} \sum\limits_{u\sim v} d_u
\end{align}
for all $v\in V$. This condition is also very restrictive, and it is only fulfilled by graphs with strong symmetries and by specific choices for the parameters $\alpha, \beta, \gamma$. Examples include the complete graph for $\alpha = \beta$ (since in that case there are no possible $\gamma$ transitions) and, as we see in Theorem~\ref{thm:statdistonregulargraph}, regular graphs with $\alpha = \beta = \gamma$.
Overall, these examples demonstrate that node2vec random walks almost never have uniform stationary distribution when $\beta\neq\gamma$, even when the underlying graph is regular.
\begin{figure}[tbp]
    \centering
    \includegraphics[width=0.25\linewidth]{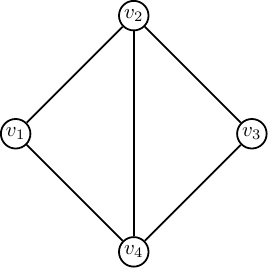}
    \caption{Example for a non-regular graph where~\eqref{eq:antireversibility} is fulfilled}
    \alt{Illustration of a four clique with one missing edge.}
    \label{fig:clique with missing edge}
\end{figure}

We conclude this section with an example of a non-regular graph that satisfies~\eqref{eq:antireversibility}. In Section~\ref{sec:results}, we saw that the weighted Eulerianity~\eqref{eq:inandoutwedgesequal} (which is equivalent to the graph being regular) implies that~\eqref{eq:antireversibility} holds but this shows that the opposite is not true.
\begin{example}[Wedge directed detailed balance do not imply regularity]\label{ex:clique with missing edge}
Although the wedge directed detailed balance condition \eqref{eq:antireversibility} is closely related to symmetries of the underlying graph, it is not equivalent to the weighted Eulerianity condition \eqref{eq:inandoutwedgesequal}. 
We illustrate this distinction with the graph shown in Figure~\ref{fig:clique with missing edge}, which consists of a $4$-clique with a single edge removed.

The graph is not regular and therefore, by Theorem~\ref{thm:statdistonregulargraph}, the weighted Eulerianity condition is violated. Nevertheless, we show that wedge directed detailed balance still holds in this case. By symmetry, the wedges of the graph fall into seven distinct equivalence classes, which can be represented by
\begin{nalign}
    &w_1 = (v_4,v_2,v_4), w_2 = (v_2,v_3,v_2), w_3 = (v_1,v_2,v_1), w_4 = (v_1,v_2,v_3) \\
    &w_5 = (v_1,v_2,v_4), w_6 = (v_2,v_1,v_4), w_7 = (v_2,v_4,v_3) 
\end{nalign}
By solving the invariant equation $\hpi = \hpi \hP$, we obtain the stationary distribution
\begin{nalign}
    \hpi(w_1) &= \hpi(w_3) = \frac{\alpha}{Z}, \\
    \hpi(w_2) &= \frac{\alpha + \beta +\gamma}{\alpha + \beta}\frac{\alpha}{Z}, \\
    \hpi(w_4) &= \frac{\gamma}{Z}, \\
    \hpi(w_5) &= \hpi(w_7) = \frac{\beta}{Z}, \\
    \hpi(w_6) &= \frac{\alpha + \beta +\gamma}{\alpha + \beta}\frac{\beta}{Z},
\end{nalign}
where $Z =10\alpha +12\beta +8\gamma$.

For comparison, on a regular $4$‑clique~\eqref{eq:antireversibility} holds because of Theorem~\ref{thm:statdistonregulargraph}. In the present graph, the missing edge reduces the number of outgoing wedges for $w_2$ and $w_6$ from three to two. To illustrate why~\eqref{eq:antireversibility} still holds, we calculate
\begin{nalign}
    \hpi(w_4)\hP(w_4,w_2) &= \frac{\gamma}{Z} \frac{\alpha}{\alpha + \beta} \\
    \hpi(w_2)\hP(-w_2,-w_4) &= \frac{\alpha + \beta +\gamma}{\alpha + \beta}\frac{\alpha}{Z} \frac{\gamma}{\alpha + \beta + \gamma} = \frac{\gamma}{Z} \frac{\alpha}{\alpha + \beta}.
\end{nalign}
As we can see here, the different denominators in $\hP(w_4,w_2)$ and $\hP(-w_2,-w_4)$ get compensated by the prefactor $\frac{\alpha+\beta+\gamma}{\alpha+\beta}$ in $\hpi(w_2)$, thereby ensuring that~\eqref{eq:antireversibility} is satisfied. The same principle applies to all other possible transitions. Thus, wedge directed detailed balance can be satisfied in non-regular graphs.

On a side note, also $\hpi(w)=\hpi(-w)$,~\eqref{eq:antireversibilityedges} and $\vpi(e) = \vpi(-e)$ are satisfied in this graph.
\end{example}

\section{Proofs of ergodicity, regularity and balance conditions}\label{ErgRegProofs}
\begin{proof}[Proof of Theorem~\ref{thm:ergodicity}]
    Let $G$ be connected and $\beta,\gamma > 0$.
    
    $(i)$ If $\alpha > 0$, there exists $m\in \NN$ such that
    \begin{align}
        \PP(\hY_m = w'|\hY_1 = w) > 0,
    \end{align}
    for all $w,w'\in W$, i.e., $\hP$ is irreducible. Now, let $G$ also contain a triangle, i.e., $\{ r,u,v \} \subset G$. Then, we compute the period of the directed wedge $w = (r,u,v)$. Observe the following transitions
    \begin{nalign}
        &\hY_1 = (r,u,v), \hY_2 = (u,v,r), \hY_3 = (v,r,u), \hY_4 = (r,u,v) \text{ and} \\ 
        &\hY_1 = (r,u,v), \hY_2 = (u,v,u), \hY_3 = (v,u,r), \hY_4 = (u,r,u), \hY_5 = (r,u,v). 
    \end{nalign}
    Because $\alpha, \beta, \gamma > 0$, these transitions all have positive probability. Thus, it is possible to transition from $w$ to $w$ in 3 and 4 steps. Since $\gcd(3,4) = 1$, $\hP$ is aperiodic.
    
    $(ii)$ If $\alpha = 0$, $d_v \ge 2 ~\forall v\in V$ and $\max_{v\in V} d_v > 2$, then $G$ is not a cycle, so the random walk cannot get stuck, neither in nodes of degree one nor in one direction of a cycle, meaning $\hP$ is irreducible. 

    For the rest of the proof, let $G$ be also finite and $\hP$ be irreducible and aperiodic. Since $G$ is finite, $\hY$ is positive recurrent and thus there exists a unique probability measure $\hpi$ on $W$ such that $\hpi = \hpi \hP$. 

    For equation~\eqref{eq:simplified invariant equation wedges}, we apply the definition of the transition probabilities $\hP$ in~\eqref{eq:wedge transition definition} to the invariant equation $\hpi = \hpi \hP$ to obtain
    \begin{align}
        \hpi(w) = \sum\limits_{w'\in W} \hpi(w') \hP(w',w) = \sum\limits_{w'\in \mathrm{IN}(w)} \hpi(w') p(w) = p(w) \sum\limits_{w'\in \mathrm{IN}(w)} \hpi(w').
    \end{align}
    For the last statement of the theorem, the aperiodicity of $\hP$ implies
    \begin{align}\label{eq:limitrep for pihat}
        \hpi(w') = \lim\limits_{n \to \infty} \hP^n(w'',w')
    \end{align}
    for all $w'\in W$ and for all starting wedges $w'' \in W$. Now, let $v \in V$ and $\{u_1,u_2\}\in \bE$. Using the definition of the $n$-step transition~\eqref{eq:nsteptransition} and representation~\eqref{eq:limitrep for pihat}, we obtain
    \begin{nalign}\label{eq:pullback existence}
        \pi(v) &\coloneqq \lim\limits_{n \to \infty} \PP(X_n = v| X_2 = u_2,X_1 = u_1) \\
        &= \lim\limits_{n \to \infty} \sum\limits_{w\in \mathrm{IN}(v)} \hP^{n-2}((u_1,u_2,\cdot),w) \\
        &= \sum\limits_{w\in \mathrm{IN}(v)} \hpi(w)
    \end{nalign}
    which implies the existence of the limit and the desired identity~\eqref{eq:pullback wedges to nodes}.
\end{proof}
\begin{proof}[Proof of Lemma~\ref{lem:relation inv measures on E and W}]
    Let $G$ be connected and let $\vP$ and $\hP$ be irreducible and aperiodic. 
    The proof of claim~\eqref{eq:pullback wedges to edges} is analogous to the above proof for~\eqref{eq:pullback wedges to nodes}. \\
    For claim~\eqref{eq:relation edges to wedges}, let $x,y,z\in V$ such that $w = (x,y,z)\in W$. By using the definition of conditional probability, we obtain
    \begin{nalign}
        \PP(X_{n+2}=z,X_{n+1}=y,X_n=x) &= \PP(X_{n+2}=z \mid X_{n+1}=y,X_n=x) \PP(X_{n+1}=y,X_n=x) \\
        &= \PP(X_{n+2}=z, X_{n+1}=y \mid X_{n+1}=y,X_n=x) \PP(X_{n+1}=y,X_n=x).
    \end{nalign}
    Now, let $n \to \infty$ on both sides to get
    \begin{align}
        \hpi((x,y,z)) = \vP((x,y),(y,z)) \vpi((x,y)).
    \end{align}
    Rewriting this in terms of $w$ yields
    \begin{align}
        \hpi(w) = \vP(e_1(w),e_2(w)) \vpi(e_1(w))
    \end{align}
    which is the wanted result.
\end{proof}
For the proof of Theorem~\ref{thm:statdistonregulargraph} we need the following lemma which states a relation between the in- and out-neighbors of two concatenated wedges.
\begin{lemma}\label{lem:wedgefact}
    Let $G$ be a connected graph. Then,
    \begin{align}
         \sum\limits_{w'\in \mathrm{ON}(w_2)} \lambda(w') = \sum\limits_{w'\in \mathrm{IN}(-w_1)} \lambda(w') \quad \forall w_2\in W,w_1 \in \mathrm{IN}(w_2).\label{eq:wedgefact2}
    \end{align}
\end{lemma}
\begin{proof}
    Let $G$ be connected and $w_2\in W, w_1 \in \mathrm{IN}(w_2).$ Name $w_2 =(r,s,t)$ and $w_1 = (p,r,s)$ for $p,r,s,t \in V$. Then $-w_1 = (s,r,p)$ and 
    \begin{align}
        \mathrm{ON}(w_2) = \{ (r,s,\cdot): (s,\cdot) \in \vE \} \text{ and } \mathrm{IN}(-w_1) = \{ (\cdot,s,r): (\cdot,s) \in \vE \}.
    \end{align}
    Together with $\lambda(w) = \lambda(-w)$, this implies
    \begin{align}
        \sum\limits_{w'\in \mathrm{ON}(w_2)} \lambda(w') = \sum\limits_{w'\in \mathrm{IN}(-w_1)} \lambda(-w') = \sum\limits_{w'\in \mathrm{IN}(-w_1)} \lambda(w').
    \end{align}
\end{proof}
\begin{proof}[Proof of Theorem~\ref{thm:statdistonregulargraph}]
    Let $G$ be a finite connected graph. 
    
    For the equivalence of $G$ regular and the weighted Eulerianity, we begin by showing that the amount ingoing and outgoing wedges for fixed $w \in W$ are the same for $\alpha$ and $\beta$ wedges, respectively. 
    Observe that for $w\in W$ with $w = (x,y,z)$ there is exactly one ingoing and one outgoing $\alpha$ wedge such that
    \begin{align}\label{eq:alphainandoutequality}
        |\{w'\in \mathrm{ON}(w)\cap W_{-}\}| = 1 = |\{w'\in \mathrm{IN}(w)\cap W_{-}\}|.
    \end{align}
    The geometry of the triangle yields that for each ingoing $\beta$ wedge there is an outgoing $\beta$ wedge on the other side of the triangle, i.e.
    \begin{nalign}
          |\{w'\in \mathrm{ON}(w)\cap W_{\Delta}\}| &= |\{(x,y,p): (x,y),(y,p),(p,x)\in \vE, p\in V\}| \\
        &= |\{(p,x,y): (x,y),(y,p),(p,x)\in \vE, p\in V\}| \\
        &= |\{w'\in \mathrm{IN}(w)\cap W_{\Delta}\}|\label{eq:betainandoutequality}
    \end{nalign}
    for all $w \in W$.
    Next, we find an equivalent formulation of equation \eqref{eq:inandoutwedgesequal}. Using the disjoint composition $W= W_{\Delta} \cup W_{\Lambda} \cup W_{-}$, equation \eqref{eq:inandoutwedgesequal} is equivalent to
    \begin{nalign}
          &\sum\limits_{w'\in \mathrm{ON}(w)\cap W_{\Delta}}\beta + \sum\limits_{w'\in \mathrm{ON}(w)\cap W_{\Lambda}}\gamma + \sum\limits_{w'\in \mathrm{ON}(w)\cap W_{-}}\alpha\\
        = &\sum\limits_{w'\in \mathrm{IN}(w)\cap W_{\Delta}}\beta + \sum\limits_{w'\in \mathrm{IN}(w)\cap W_{\Lambda}}\gamma + \sum\limits_{w'\in \mathrm{ON}(w)\cap W_{-}}\alpha.
    \end{nalign}
    Since by~\eqref{eq:alphainandoutequality} and~\eqref{eq:betainandoutequality} the two sums over the $\alpha$ and the two sums over $\beta$ are the same, we obtain that~\eqref{eq:inandoutwedgesequal} is equivalent to
    \begin{align}
          \sum\limits_{w'\in \mathrm{ON}(w)\cap W_{\Lambda}} \gamma = \sum\limits_{w'\in \mathrm{IN}(w)\cap W_{\Lambda}} \gamma
    \end{align}
    which is equivalent to
    \begin{align}\label{eq:gammainandoutequality}
          |\{w'\in \mathrm{ON}(w)\cap W_{\Lambda}\}| = |\{w'\in \mathrm{IN}(w)\cap W_{\Lambda}\}|
    \end{align}
    for all $w\in W$. 
    Now, we can prove the statement. We begin by showing that~\eqref{eq:inandoutwedgesequal} implies that $G$ is regular.
    If we consider a wedge $w_1 = (u,v,x)\in W$, we can characterize the degrees of $u$ and $v$ by
    \begin{nalign} \label{eq:degreeformulas}
        &d_u = |\{ w'\in \mathrm{IN}(w_1)\cap W_{\Lambda} \}| + |\{w'\in \mathrm{IN}(w_1)\cap W_{\Delta}\}| + 1 \\
        &d_v = |\{ w'\in \mathrm{ON}(w_1)\cap W_{\Lambda} \}| + |\{w'\in \mathrm{ON}(w_1)\cap W_{\Delta}\}| + 1.
    \end{nalign}
    In both cases, the $+1$ comes from the connection $(u,v)$ which is represented by the corresponding flat wedge. Since by~\eqref{eq:betainandoutequality} and~\eqref{eq:gammainandoutequality} the first two above terms are equal, we obtain that $d_u = d_v$. Doing the same for $w_2 = (u,v,y) \in \vE$ implies that $d_v=d_y$. Since $G$ is connected, we can continue this procedure to obtain $d_i = d_j$ for all $i,j\in V$. 
    
    The other direction follows similarly: Since~\eqref{eq:betainandoutequality} holds in general, if we have $d_u = d_v$, we need to fulfill
    \begin{align}\label{eq:rest}
        |\{ w'\in \mathrm{IN}(w_1)\cap W_{\Lambda} \}| = |\{ w'\in \mathrm{ON}(w_1)\cap W_{\Lambda} \}|
    \end{align}
    as well because otherwise we get a contradiction in~\eqref{eq:degreeformulas}. As shown above,~\eqref{eq:rest} is equivalent to~\eqref{eq:inandoutwedgesequal} and the claimed equivalence follows. 
    
    For the rest of the proof, let $\hP$ be irreducible and assume that $G$ satisfies~\eqref{eq:inandoutwedgesequal}. For proving the expression of the stationary distribution, let $w \in W$ and define $Z \coloneqq \sum_{w'\in W} \lambda(w')$. We want to find a solution $\hpi$ for the invariant identity $\hpi = \hpi \hP$. We write out the transition probability using equation~\eqref{eq:defwedgetransition} and reorder the terms to get
    \begin{align}
        \hpi(w) = \sum\limits_{w'\in \mathrm{IN}(w)} \hpi(w') \hP(w',w) 
        = \sum\limits_{w'\in \mathrm{IN}(w)} \hpi(w') \frac{\lambda(w)}{\sum\limits_{\Tilde{w}\in \mathrm{ON}(w)}\lambda(\Tilde{w})} 
        = \lambda(w) \frac{\sum\limits_{w'\in \mathrm{IN}(w)} \hpi(w')}{\sum\limits_{\Tilde{w}\in \mathrm{ON}(w)} \lambda(\Tilde{w})}.
    \end{align}
    Now we can use the weighted Eulerianity~\eqref{eq:inandoutwedgesequal} to obtain
    \begin{align}
        \hpi(w) = \lambda(w) \frac{\sum\limits_{w'\in \mathrm{IN}(w)} \hpi(w')}{\sum\limits_{\Tilde{w}\in \mathrm{IN}(w)} \lambda(\Tilde{w})}
    \end{align}
    This equation is solved by $\hpi(w) = \frac{\lambda(w)}{Z}$ which is the unique stationary distribution. 
    
    For showing that $\hpi$ satisfies~\eqref{eq:antireversibility}, let $w_2 \in W$ and $ w_1 \in \mathrm{IN}(w_2)$. The expression of the stationary distribution, the definition of the wedge transitions~\eqref{eq:defwedgetransition} and Lemma~\ref{lem:wedgefact} yield
    \begin{align}\label{eq:antireversibilityproof}
        \hpi(w_1) \hP(w_1,w_2) &= \frac{\lambda(w_1)}{Z} \frac{\lambda(w_2)}{\sum\limits_{w'\in \mathrm{ON}(w_2)}\lambda(w')} = \frac{\lambda(w_1)}{Z} \frac{\lambda(w_2)}{\sum\limits_{w'\in \mathrm{IN}(-w_1)}\lambda(w')}.
    \end{align}
    By switching the numerators and using~\eqref{eq:wedgefact1} and the weighted Eulerianity~\eqref{eq:inandoutwedgesequal}, we obtain
    \begin{nalign}
        \frac{\lambda(w_1)}{Z} \frac{\lambda(w_2)}{\sum\limits_{w'\in \mathrm{IN}(-w_1)}\lambda(w')}
        = \frac{\lambda(-w_2)}{Z} \frac{\lambda(-w_1)}{\sum\limits_{w'\in \mathrm{ON}(-w_1)}\lambda(w')} 
        = \hpi(-w_2)\hP(-w_2,-w_1),
    \end{nalign}
    which shows~\eqref{eq:antireversibility}.
\end{proof}
\begin{proof}[Proof of Corollary~\ref{cor:anti-rev on edges}]
    Let $G$ be a finite regular graph and let $e,e'\in \vE$ be distinct edges. We begin by showing~\eqref{eq:antireversibilityedges}. By using Theorem~\ref{thm:statdistonregulargraph} and the pullback~\eqref{eq:pullback wedges to edges}, we obtain the stationary distribution
    \begin{align}
        \vpi(e) = \sum\limits_{w \in \mathrm{IN}(e)} \frac{\lambda(w)}{Z}
    \end{align}
    for all $e\in \vE$ where $Z \coloneqq \sum_{w'\in W}\lambda(w')$. We use this and express $\vP(e,e')$ in terms of the wedge transition probability~\eqref{eq:defwedgetransition} to get
    \begin{align}
        \vpi(e)\vP(e,e') &= \sum\limits_{w \in \mathrm{IN}(e)} \frac{\lambda(w)}{Z} \frac{\lambda((e,e'))}{\sum\limits_{\tilde{w} \in \mathrm{ON}((e,e'))} \lambda(\tilde{w})}
    \end{align}
    With the weighted Eulerianity~\eqref{eq:inandoutwedgesequal}, we compute
    \begin{nalign}
         \sum\limits_{w \in \mathrm{IN}(e)} \frac{\lambda(w)}{Z} \frac{\lambda((e,e'))}{\sum\limits_{\tilde{w} \in \mathrm{ON}((e,e'))} \lambda(\tilde{w})}
        = \sum\limits_{w \in \mathrm{IN}(e)} \frac{\lambda(w)}{Z} \frac{\lambda((e,e'))}{\sum\limits_{\tilde{w} \in \mathrm{IN}((e,e'))} \lambda(\tilde{w})} 
        = \frac{\lambda((e,e'))}{Z}.
    \end{nalign}
    To start from the right-hand side of~\eqref{eq:antireversibilityedges}, we use similar arguments as above to get
    \begin{align}
        \sum\limits_{\tilde{w} \in \mathrm{ON}((-e',-e))} \lambda(\tilde{w}) = \sum\limits_{\tilde{w} \in \mathrm{IN}(e')} \lambda(-\tilde{w}) = \sum\limits_{\tilde{w} \in \mathrm{IN}(e')} \lambda(\tilde{w})
    \end{align}
    which yields
    \begin{align}
        \vpi(e')\vP(-e',-e) = \sum\limits_{w \in \mathrm{IN}(e')} \frac{\lambda(w)}{Z} \frac{\lambda((-e',-e))}{\sum\limits_{\tilde{w} \in \mathrm{ON}((-e',-e))} \lambda(\tilde{w})}
        = \frac{\lambda((-e',-e))}{Z}.
    \end{align}
    By equation~\eqref{eq:wedgefact1} we have
    \begin{align}
        \lambda((e,e')) = \lambda((-e',-e))
    \end{align}
    which proves the equality in~\eqref{eq:antireversibilityedges}. 
    To 
    prove~\eqref{eq:regular opposite edges equal stat}, we use the pullback~\eqref{eq:pullback wedges to edges}, Theorem~\ref{thm:statdistonregulargraph} and a symmetry argument to obtain
    \begin{nalign}
        \vpi(e) = \sum\limits_{w\in \mathrm{IN}(e)} \hpi(w) = \sum\limits_{w\in \mathrm{IN}(e)} \frac{\lambda(w)}{Z} 
        = \sum\limits_{w\in \mathrm{ON}((-e,\cdot))} \frac{\lambda(w)}{Z}.
    \end{nalign}
    The weighted Eulerianity~\eqref{eq:inandoutwedgesequal} and doing above steps in reverse order yield
    \begin{nalign}
        \sum\limits_{w\in \mathrm{ON}((-e,\cdot))} \frac{\lambda(w)}{Z}
        = \sum\limits_{w\in \mathrm{IN}((-e,\cdot))} \frac{\lambda(w)}{Z}
        = \sum\limits_{w\in \mathrm{IN}(-e)} \frac{\lambda(w)}{Z} 
        = \sum\limits_{w\in \mathrm{IN}(-e)} \hpi(w) = \vpi(-e).
    \end{nalign}
\end{proof}
\begin{proof}[Proof of Lemma~\ref{lem:directed detailed balance relations}]
    Let $G$ be connected. Assume that $\vpi$ fulfills~\eqref{eq:antireversibilityedges}. \\
    $(1a)$ Let $e'\in \vE$. By using~\eqref{eq:antireversibilityedges} and by relabeling the sum, we obtain
    \begin{nalign}
        (\vpi\vP)(e') &= \sum\limits_{e\in \vE} \vpi(e) \vP(e,e') = \sum\limits_{e\in \vE} \vpi(e') \vP(-e',-e) \\
        &= \vpi(e') \sum\limits_{e\in \vE} \vP(-e',-e) \\
        &= \vpi(e') \sum\limits_{-e\in \vE} \vP(-e',-e) = \vpi(e').
    \end{nalign}
    $(1b)$ Let $\alpha >0$. Set $e' \coloneqq -e$ in~\eqref{eq:antireversibilityedges} to get
    \begin{align}\label{eq:edge and opposite edge balance}
        \vpi(e)\vP(e,-e) = \vpi(-e)\vP(e,-e).
    \end{align}
    Since $\alpha >0$, also $\vP(e,-e)>0$ and~\eqref{eq:edge and opposite edge balance} implies
    \begin{align}
        \vpi(e) = \vpi(-e).
    \end{align}
    $(1c)$ Now, let $\alpha = 0$, let $e_0,e_1,e_2,\ldots,e_{n-1},e_n,-e_0$ be a path of length $n+2$ such that $e_0,e_i\in \vE$, $\vP(e_{i-1},e_i)>0$ for all $i\in [n]$ and $\vP(e_n,-e_0)>0$. By using~\eqref{eq:antireversibilityedges} repeatedly, we obtain
    \begin{align}
        \vpi(e_0) \vP(e_0,e_1)\vP(e_1,e_2)\cdots \vP(e_n,-e_0) = \vpi(-e_0) \vP(e_0,-e_n)\vP(-e_n,-e_{n-1})\cdots \vP(-e_1,-e_0)
    \end{align}
    which implies that $\vpi(e_0) = \vpi(-e_0)$ holds if and only if
    \begin{align}\label{eq:cycle probabilities}
        \vP(e_0,e_1)\vP(e_1,e_2)\cdots \vP(e_n,-e_0) = \vP(e_0,-e_n)\vP(-e_n,-e_{n-1})\cdots \vP(-e_1,-e_0)
    \end{align}
    holds. Filling in the probabilities in~\eqref{eq:cycle probabilities} and setting $e_{n+1} = -e_0$ for notational simplicity yields
    \begin{align}
        \prod\limits_{i=0}^n\frac{\lambda((e_{i},e_{i+1}))}{\sum_{w\in ON(e_{i},\cdot)} \lambda(w)} = \prod\limits_{i=0}^n\frac{\lambda((-e_{n-i+1},-e_{n-i}))}{\sum_{w\in ON(-e_{n-i+1},\cdot)} \lambda(w)}.
    \end{align}
    Since $\lambda(w)=\lambda(-w)$ for all $w\in W$, the product of the numerators on both sides are equal. Furthermore, it holds $-e_{n+1}=e_0$, so the denominator for $i=0$ is equal on both sides. Using this implies that~\eqref{eq:cycle probabilities} is equivalent to
    \begin{align}
        \prod\limits_{i=1}^n\sum_{w\in ON(e_{i},\cdot)} \lambda(w) = \prod\limits_{i=1}^n \sum_{w\in ON(-e_{n-i+1},\cdot)} \lambda(w).
    \end{align}
    By using $ON((-e,\cdot))= IN((e,\cdot))$ for all $e\in \vE$ and by reordering the product on the right-hand side, we obtain that~\eqref{eq:cycle probabilities} is equivalent to
    \begin{align}
        \prod\limits_{i=1}^n\sum_{w\in ON(e_{i},\cdot)} \lambda(w) = \prod\limits_{i=1}^n \sum_{w\in IN(e_i,\cdot)} \lambda(w).
    \end{align}
    which is the desired statement. \\
    $(1d)$ Let $\vpi(e) = \vpi(-e)$ hold for all $e\in \vE$. Applying~\eqref{eq:antireversibilityedges} twice yields
    \begin{align}\label{eq:twostep EDB}
        \vpi(e)\vP(e,e')\vP(e',e'') = \vpi(e')\vP(-e',-e)\vP(e',e'') = \vpi(e'')\vP(-e'',-e')\vP(-e',-e).
    \end{align}
    Remembering the relation~\eqref{eq:relation edges to wedges}, we set
    \begin{align}
        \hpi(w) \coloneqq \vpi(e_1(w))\vP(e_1(w),e_2(w))
    \end{align}
    for all $w\in W$ and note that for $w\in \mathrm{IN}(w')$ we have
    \begin{align}
        \hP(w,w') = \vP(e_1(w'),e_2(w')).
    \end{align}
    Using~\eqref{eq:antireversibilityedges} and $\vpi(e) = \vpi(-e)$, for $w\in W$, we compute
    \begin{nalign}\label{eq:EDB implies WDB, stat dist balance}
        \hpi(w) &= \vpi(e_1(w))\vP(e_1(w),e_2(w)) = \vpi(e_2(w))\vP(-e_2(w),-e_1(w)) \\
        &= \vpi(-e_2(w))\vP(-e_2(w),-e_1(w)) = \hpi(-w).
    \end{nalign}
    Let $w,w'\in W$ such that $w\in \mathrm{IN}(w')$. Putting all of that together, we obtain
    \begin{align}\label{eq:edge representation EDB WDB proof}
        \hpi(w) \hP(w,w') = \vpi(e_1(w))\vP(e_1(w),e_2(w)) \vP(e_1(w'),e_2(w')).
    \end{align}
    Since $w\in \mathrm{IN}(w')$, it holds $e_2(w) = e_1(w')$ and we can use~\eqref{eq:twostep EDB} and $\vpi(e) = \vpi(-e)$ to get
    \begin{nalign}
        \hpi(w) \hP(w,w') &= \vpi(e_2(w'))\vP(-e_2(w'),-e_1(w'))\vP(-e_2(w),-e_1(w)) \\
        &= \vpi(-e_2(w'))\vP(-e_2(w'),-e_1(w'))\vP(-e_2(w),-e_1(w))\\
        &=\hpi(-w')\hP(-w',-w)
    \end{nalign}
    Together with~\eqref{eq:EDB implies WDB, stat dist balance}, we obtain~\eqref{eq:antireversibility}.

    Now, assume that $\hpi$ fulfills~\eqref{eq:antireversibility}. \\
    $(2a)$ The proof is analog to the proof of $(1a)$. \\
    $(2b)$ Name $w = (u,v,x)$ and $w_\alpha = (v,x,v)$. Since $\alpha >0$, the path $(w,w_\alpha,-w)$ has positive probability. Using~\eqref{eq:antireversibility} twice and reordering the terms, we compute
    \begin{align}
        \hpi(w)\hP(w,w_\alpha)\hP(w_\alpha,-w) &= \hpi(w_\alpha)\hP(-w_\alpha,-w)\hP(w_\alpha,-w) 
        = \hpi(-w)\hP(w,-w_\alpha)\hP(-w_\alpha,-w).
    \end{align}
    Note that $w_\alpha = (v,x,v) = -w_\alpha$, so we obtain
    \begin{align}
        \hpi(w)\hP(w,w_\alpha)\hP(w_\alpha,-w) = \hpi(-w)\hP(w,w_\alpha)\hP(w_\alpha,-w).
    \end{align}
    Since by construction all probabilities are positive, $\hpi(w) = \hpi(-w)$ follows.
\end{proof}
\section{Proof of recurrence on infinite graphs}\label{RecProofs}

The goal of this section is to show that node2vec random walks are recurrent on regular graphs if and only if the SRW is recurrent on that graph. We employ ideas from electrical network theory to compare the recurrence properties of different Markov chains, like the following lemma.
\begin{lemma}(follows from Thm. 2.17 in~\cite{lyons_probability_2017})\label{lem:recurrence equivalence with different conductances}
    Let $G = (V,\bE)$ be an infinite connected undirected graph and let $X_1,X_2$ be reversible Markov chains on $G$ with corresponding transition matrices $P_1,P_2$ that fulfill $0 < P_i(v_1,v_2) \le 1$ for all $(v_1,v_2) \in \vE$ and for $k=1,2$. Then,
    \begin{align}
        X_1 \text{ is recurrent on } G \Longleftrightarrow X_2 \text{ is recurrent on } G.
    \end{align}
\end{lemma}
This lemma is a very useful tool to show recurrence by comparing it to a different random walk. We will use~\eqref{eq:antireversibilityedges} to construct an auxiliary Markov chain that is reversible for which we can apply this Lemma. This construction closely follows the construction in the proof of Theorem 1 in~\cite{hermon_reversibility_2019}.

We divide the proof in the cases $\alpha>0$ and $\alpha = 0$ because they require different assumptions on the graph and a different construction of the auxiliary chain. As the case $\alpha>0$ allows the node2vec random walks to backtrack, it is easier there to construct paths that return to a previous state than in the case $\alpha = 0$.
\begin{proof}[Proof of Theorem~\ref{thm:recurrence}]
    For the case $\alpha > 0$, let $\vP$ be the kernel of the edge Markovian representation of the node2vec random walk $X$ with parameters $\alpha, \beta, \gamma > 0$. Let $Y \coloneqq (Y_n)_{n=0}^\infty$ be the Markov chain corresponding to the lazy version of $\vP$, of which we denote the kernel by $\vP_L = \frac{1}{2}(I+\vP)$. We can interpret $Y$ in the following way: Given $Y_n$ the chain picks a candidate $T_{n+1}$ for $Y_{n+1}$ according to $\vP$ and then flips a fair coin to decide if the candidate is accepted. If not, then $Y_{n+1}=Y_n$.

    Let $p = \min_{e\in \vE}\vP(e,-e)$ be the minimum probability of $Y$ to do a backtracking step. As $\alpha>0$ and the degrees of $G$ are bounded, this minimum exists. 
    
    To generate $Y$ and $T_1,T_2,\ldots$, we define i.i.d. Bernoulli($p$) random variables $\xi_0, \xi_1,\ldots$ such that for $n\ge 0$ we have $T_{n+1} = -Y_n$ whenever $\xi_n = 1$ (and also possibly when $\xi_n = 0$). Thus, at every time step, $Y$ backtracks at least with probability $p$ (but $Y$ may backtrack also when this Bernoulli flip fails). The proof of the construction of the $\xi$ can be found in proof of Theorem 1 (ii) in~\cite{hermon_reversibility_2019}.

    We now construct a Markov chain $Q$ from $Y$ of which the reversibility properties are easier to analyze, as it is undirected and reversible.  To be more precise, fix a directed edge $e = (u_0,v_0)\in \vE$. Consider the case where the initial distribution of $Y$ is uniform on $\{ e,-e \}$. Let $\be_0 = \{ u_0,v_0 \}$ be the undirected version of $e_0$ and $\tau_0=-1$. Inductively, we define $\tau_{n+1}\coloneqq \inf \{ t > \tau_n: \xi_{t-1} = 1 \}$ and $\be_n \coloneqq \{ u_n,v_n \}$ where $u_n$ and $v_n$ are the endpoints of $T_{\tau_n+1}$. Then, the process $Q \coloneqq (\be_n)_{n=0}^\infty$ on the undirected edges $\bE$ is given by $Q_n = \be_n$ for all $n\ge 0$. Informally, $Q$ tracks the steps of $Y$ at which $Y$ backtracks. In the rest of the proof, we show that:
    \begin{itemize}
        \item[(i)] $Q$ is a reversible Markov chain with respect to $\bpi(\{ e,-e \}) \coloneqq \vpi(e) + \vpi(-e)$. In other words, its transition kernel $\bK$ satisfies
        \begin{align}
            \bpi(\{ e,-e \}) \bK(\{ e,-e \},\{ e',-e' \}) = \bpi(\{ e',-e' \}) \bK(\{ e',-e' \},\{ e,-e \})
        \end{align}
        for all $e,e' \in \vE$.
        \item[(ii)] $X$ is recurrent if and only if $Q$ is recurrent.
        \item[(iii)] SRW on $G$ is recurrent if and only if $Q$ is recurrent.
    \end{itemize}
    \textit{Step (i).} Consider the transition kernel $K$ on $\vE$ given by
    \begin{align}\label{eq:Kdefinition}
        K(e,e') = \begin{cases}
            (1-p)^{-1} \vP(e,e'), &\text{if } e \neq -e' \\
            (1-p)^{-1} (\vP(e,-e)-p), &\text{if } e = -e'.
        \end{cases}
    \end{align}
    Informally, $K$ is the transition kernel of $Y$ conditionally on the Bernoulli $p$ flip failing. We define the lazy version of $K$ as $K_L(e,e') = \frac{1}{2}(I(e,e')+K(e,e'))$ where $I$ is the identity kernel, so that the walk remains at its current edge with probability $1/2$.
    
    We now show that $K_L(e,e')$ satisfies a directed detailed balance. Let $e \neq -e'$. Denote by $P_{e_0}$ the law of the Markov chain $Y$ starting in $e_0$. By using~\eqref{eq:antireversibilityedges} repeatedly, we can show that
    \begin{nalign}\label{eq:antirevsteps}
        \vpi(e_0) P_{e_0}(Y_1 = e_1,\ldots, Y_n = e_n) = \vpi(e_n) P_{-e_n}(Y_1 = -e_{n-1},\ldots, Y_n = -e_0).
    \end{nalign}
    Let $\gamma_n(e,e')$ denote the set of paths from $e$ to $e'$ in $n$ steps. Then~\eqref{eq:antirevsteps} yields
    \begin{nalign}\label{eq:vPn reversibility}
        \vpi(e)\vP^n(e,e') &= \sum\limits_{\gamma_n(e,e')} \vpi(e) P_{e}(Y_1 = e_1,\ldots, Y_n = e') \\
        &= \sum\limits_{\gamma_n(e,e')} \vpi(e') P_{-e'}(Y_1 = -e_{n-1},\ldots, Y_n = -e) \\
        &= \vpi(e')\vP^n(-e',-e),
    \end{nalign}
    for all $n\ \in \NN$. Combining this with the definition of $K$ in~\eqref{eq:Kdefinition} gives
    \begin{nalign}
        \vpi(e) K^n(e,e') &= (1-p)^{-n}\vpi(e)\vP^n(e,e') \\
        &= (1-p)^{-n}\vpi(e')\vP^n(-e',-e) = \vpi(e') K^n(-e',-e),
    \end{nalign}
    for all $n \in \NN$. We can extend this directed detailed balance to the lazy version of $K$ by computing
    \begin{nalign}\label{eq:lazyK reversibility case 1}
        \vpi(e) K^n_L(e,e') &= \sum\limits_{i=0}^n 2^{-n} \binom{n}{i}\vpi(e)K^i(e,e') \\
        &= \sum\limits_{i=0}^n 2^{-n} \binom{n}{i}\vpi(e')K^i(-e',-e) \\
        &= \vpi(e') K^n_L(-e',-e).
    \end{nalign}
    This shows that $K_L$ satisfies the directed detailed balance when $e\neq -e'$. 
    
    We now study the case $e= -e'$. Since $\alpha > 0$, we can apply Lemma~\ref{lem:directed detailed balance relations}$(1b)$ to get $\vpi(e) = \vpi(-e)$. This yields
    \begin{nalign}
        \vpi(e) K^n(e,e')  = \vpi(e) K^n(-e',-e) = \vpi(-e) K^n(-e',-e)  = \vpi(e') K^n(-e',-e).
    \end{nalign}
    Similar calculations as in~\eqref{eq:lazyK reversibility case 1} yield
    \begin{align}\label{eq:final lazy K reversibility}
        \vpi(e) K^n_L(e,e') = \vpi(e') K^n_L(-e',-e)
    \end{align}
    for all $e,e' \in \vE$. Thus, $K_L$ satisfies the directed detailed balance also in this case.

    We can write the kernel $\bK$ of $Q$ on the undirected edges as
    \begin{align}\label{eq:qkerneldefinition}
        \bK(\{ e,-e\},\{ e',-e'\}) \coloneqq \frac{1}{2}\sum\limits_{i=0}^\infty \sum\limits_{e_1 \in \{ e,-e\}, e_2 \in \{ e',-e'\}} K_L^i(e_1,e_2) (1-p)^ip.
    \end{align}
    Indeed, between two tracked edges of $Q$, the walk $Y$ behaves as $K_L$. Furthermore, each step of $Y$ is included in $Q$ with probability $p$. Finally, as $Y$ is a lazy walk with parameter $1/2$, every edge of $Y$ that is included in $Q$ is equally likely to be $( e,-e)$ as $ (-e,e)$, allowing to work on the undirected edges.
    
    We now show that $\bK$ is reversible. Remembering that we have $\vpi(e) = \vpi(-e)$ yields $\bpi(\{ e,-e \}) \coloneqq \vpi(e) + \vpi(-e) = 2 \vpi(e)$. We begin by using the definition~\eqref{eq:qkerneldefinition} and by writing out the inner sum to get
    \begin{nalign}
        &\bpi(\{ e,-e \}) \bK(\{ e,-e \},\{ e',-e' \}) \\
        =~& \frac{\bpi(\{ e,-e \})}{2}\sum\limits_{i=0}^\infty \sum\limits_{e_1 \in \{ e,-e\}, e_2 \in \{ e',-e'\}} K_L^i(e_1,e_2) (1-p)^ip \\
        =~& \frac{\bpi(\{ e,-e \})}{2}\sum\limits_{i=0}^\infty (1-p)^ip (K_L^i(e,e') + K_L^i(e,-e') +K_L^i(-e,e') + K_L^i(-e,-e')).
    \end{nalign}
    Then, we can pull $\bpi$ into the sum and apply~\eqref{eq:final lazy K reversibility} to obtain
    \begin{nalign}
        &\frac{\bpi(\{ e,-e \})}{2}\sum\limits_{i=0}^\infty (1-p)^ip (K_L^i(e,e') + K_L^i(e,-e') +K_L^i(-e,e') + K_L^i(-e,-e')) \\
        =~&\sum\limits_{i=0}^\infty (1-p)^ip (\vpi(e)K_L^i(e,e') + \vpi(e)K_L^i(e,-e') +\vpi(e)K_L^i(-e,e') + \vpi(e)K_L^i(-e,-e')) \\
        =~&\sum\limits_{i=0}^\infty (1-p)^ip (\vpi(e')K_L^i(e',e) + \vpi(e')K_L^i(-e',e) +\vpi(e')K_L^i(e',-e) + \vpi(e')K_L^i(-e',-e)).
    \end{nalign}
    Putting the sum back together and using the definition~\eqref{eq:qkerneldefinition} once more yields
    \begin{nalign}
        &\sum\limits_{i=0}^\infty (1-p)^ip (\vpi(e')K_L^i(e',e) + \vpi(e')K_L^i(-e',e) +\vpi(e')K_L^i(e',-e) + \vpi(e')K_L^i(-e',-e)) \\
        =~& \frac{\bpi(\{ e',-e' \})}{2}\sum\limits_{i=0}^\infty (1-p)^ip (K_L^i(e',e) + K_L^i(-e',e) +K_L^i(e',-e) + K_L^i(-e',-e)) \\
        =~&\bpi(\{ e',-e' \}) \bK(\{ e',-e' \},\{ e,-e \})
    \end{nalign}
    which shows that $\bK$ is reversible with regard to $\bpi$.
    
    \textit{Step (ii).} First, observe that $Y$ is the lazy version of $X$ on the directed edges, so their recurrence properties coincide. 
    To show that $Y$ is recurrent if and only if $Q$ is recurrent, 
    we use the fact that $Q$ is defined as a subsampling from $Y$. Let $\{e_0,e_1,\dots\}$ be the edges visited by $Y$. Now suppose that edge $e_i$ is visited infinitely often by $Y$. As every edge of $Y$ is independently included in $Q$ with probability $p>0$, this means that the probability that (the undirected version of) $e_i$ is not visited infinitely often in $Q$ is zero. Thus, whenever $Y$ is recurrent, $Q$ is also recurrent.

    Now, let $\{\be_0,\be_1,\dots\}$ be the undirected edges visited by $Q$, and suppose that edge $\be_i$ is visited infinitely often. This means that the two directed versions of $\be_i$ are also visited infinitely often in $Y$. Indeed, as $Q$ is a subsampling of $Y$, for each occurrence of $\be_i$, one of its two directed versions should appear in $Y$. As both versions appear with equal probability, this means that the probability that at least one of the two undirected versions is not visited infinitely often is zero. Thus, $Y$ is recurrent whenever $Q$ is recurrent. 
    
    \textit{Step (iii).} Every reversible chain corresponds to a network $(V,\vE,(c_e)_{e\in \vE})$ with state space $V$, edge set $\vE$ and edge weights $c_e$ for every $e\in \vE$, see e.g.~\cite{lyons_probability_2017}. $Q$ corresponds to the network $(\bE,F,(\bpi(\be)K(\be,\be'))_{\{ \be,\be' \}\in F})$, where
    \begin{align}
        F \coloneqq \{ \{ \be,\be' \}: \be,\be'\in \bE \text{ s.t. } \bpi(\be)\bK(\be,\be') > 0 \}.
    \end{align}
    Let $\Tilde{Q}$ be the $Q$ from the proof of Theorem 3 in~\cite{hermon_reversibility_2019} with kernel $W$. For all $\be,\be' \in \bE$, $0 < W(\be,\be') \le 1$ and $0 < \bK(\be,\be') \le 1$. Therefore, in both networks (corresponding to $Q$ and $\Tilde{Q}$) all nodes are at distance one. Furthermore, the state spaces of $\tilde{Q}$ and $Q$ are identical, and therefore the corresponding networks consist of the same graph, only the edge weights are different.
    By Lemma~\ref{lem:recurrence equivalence with different conductances}, this implies that $Q$ is recurrent if and only if $\Tilde{Q}$ is recurrent. The proof of Theorem 3 in~\cite{hermon_reversibility_2019} shows that $\Tilde{Q}$ is recurrent if and only if the SRW on $G$ is recurrent, so we obtain that $Q$ is recurrent if and only if the SRW on $G$ is recurrent. \\
    \ \\
    For the case $\alpha = 0$, let $\vP_0$ be the kernel of the edge Markovian representation of the node2vec random walk $X$ with parameters $\alpha = 0$ and $\beta, \gamma > 0$ on the directed edges $\vE$. By condition~\eqref{eq:condition every ball contains cycle}, there exists $M\in \NN$ such that for all $e \in \vE$ we can construct a path $e_1 = e,\ldots,e_l = -e$ of length $l\le M$ with $\vP_0(e_i,e_{i+1})>0$ for all $i \in [l-1]$. Set
    \begin{align}
    p \coloneqq \frac{1}{2(M+1)}\left(\frac{\min(\beta,\gamma)}{\min(\beta,\gamma)+(d-1)\max(\beta,\gamma)}\right)^M.
    \end{align}
    The probability to take such a path $e_1 = e,\ldots,e_l = -e$ is
    \begin{align}
    \prod\limits_{i=1}^{l-1} \vP_0(e_i,e_{i+1}) \ge \left(\frac{\min(\beta,\gamma)}{\min(\beta,\gamma)+(d-1)\max(\beta,\gamma)}\right)^M = 2(M+1)p.
    \end{align}
    This implies
    \begin{align}\label{eq:bigger than min return prob}
        \sum\limits_{i=1}^M \vP_0^i(e,-e) \ge 2p(M+1).
    \end{align}
    for all $e\in \vE$. We define $D \coloneqq \frac{1}{M+1}\sum\limits_{i=0}^M \vP_0^i$, $D_L \coloneqq \frac{1}{2}(I+D)$ its lazy version and $Z = (Z_n)_{n=0}^{\infty}$ the Markov chain corresponding to $D_L$. Using equation~\eqref{eq:bigger than min return prob}, we get $D(e,-e)\ge 2p$ and thus $D_L(e,-e)\ge p$ for all $e\in \vE$. 
    By Lemma 8.1 in~\cite{hermon_reversibility_2019}, $\vP_0$ is transient if and only if $Z$ is transient.
    
    Similar as above, we can interpret $Z$ in the following way: Given $Z_n$ the chain picks a candidate $T^0_{n+1}$ for $Z_{n+1}$ according to $D$ and then flips a fair coin in order to decide if the candidate is accepted. To generate $Z$ and $T^0_1,T^0_2,\ldots$, we define i.i.d. Bernoulli($p$) random variables $\xi_0, \xi_1,\ldots$ such that for $n\ge 0$ we have $T^0_{n+1} = -Z_n$ whenever $\xi_n = 1$ (and also possibly when $\xi_n = 0$). Using equation~\eqref{eq:bigger than min return prob}, the probability that $T^0_{n+1} = -Z_n$ is at least $p$. As before, the proof of the construction of the $\xi$ can be found in proof of Theorem 1 (ii) in~\cite{hermon_reversibility_2019}.
    
    To construct $Q$, fix a directed edge $e \in \vE$. Consider the case where the initial distribution of $Z$ is uniform on $\{ e,-e \}$. Let $\be_0 = \{ e,-e \}$ and $\tau_0=-1$. Inductively, we define $\tau_{n+1}\coloneqq \inf \{ t > \tau_n: \xi_{t-1} = 1 \}$ and $\be_n \coloneqq \{ T^0_{\tau_n+1},-T^0_{\tau_n+1} \}$. Then, the process $Q^0 \coloneqq (\be_n)_{n=0}^\infty$ on the undirected edges $\bE$ is given by $Q^0_n = e_n$ for all $n\ge 0$. We need to show that:
    \begin{itemize}
        \item[(i)] $Q^0$ is a reversible Markov chain with respect to $\bpi(\{ e,-e \}) \coloneqq \vpi(e) + \vpi(-e)$. In other words, its transition kernel $\bK$ satisfies
        \begin{align}
            \bpi(\{ e,-e \}) \bK_0(\{ e,-e \},\{ e',-e' \}) = \bpi(\{ e',-e' \}) \bK_0(\{ e',-e' \},\{ e,-e \})
        \end{align}
        for all $e,e' \in \vE$.
        \item[(ii)] $X$ is recurrent if and only if $Q^0$ is recurrent. 
        \item[(iii)] SRW on $G$ is recurrent if and only if $Q^0$ is recurrent.
    \end{itemize}
    \textit{Step (i).} Consider the transition kernel $K_0$ on $\vE$ given by
    \begin{align}
        K_0(e,e') = \begin{cases}
            (1-p)^{-1} D(e,e'), &\text{if } e \neq -e' \\
            (1-p)^{-1} (D(e,-e)-p), &\text{if } e = -e',
        \end{cases}
    \end{align}
    and its lazy version $K_{0,L}(e,e') = \frac{1}{2}(I(e,e')+K_0(e,e'))$.
    Since $\vP_0$ fulfills~\eqref{eq:antireversibilityedges}, $D$ fulfills it as well. Following the same steps as in the proof of the case $\alpha > 0$, we obtain
    \begin{align}
        \bK_0(\{ e,-e\},\{ e',-e'\}) \coloneqq \frac{1}{2}\sum\limits_{i=0}^\infty \sum\limits_{e_1 \in \{ e,-e\}, e_2 \in \{ e',-e'\}} K_{0,L}^i(e_1,e_2) (1-p)^ip
    \end{align}
    is reversible with regard to $\pi(\{ e,-e \})$.
    
    \textit{Step (ii).} This step is identical to Step (ii) in the proof of the case $\alpha > 0$.
    
    \textit{Step (iii).} This step is identical to Step (iii) in the proof of the case $\alpha > 0$ with the adapted kernel $\bK_0$ and $\tilde{Q}$ the kernel $Q$ from the proof of Theorem 1 (ii) in~\cite{hermon_reversibility_2019}.
\end{proof}

\printbibliography
\end{document}